\documentclass[12pt]{amsart}

\usepackage[utf8]{inputenc}
\usepackage[a4paper,nomarginpar]{geometry}

\usepackage[english]{babel}
\usepackage{enumitem}
\usepackage{amsmath,amsfonts,amssymb,amsthm}
\usepackage{amscd}
\usepackage{tikz}

\allowdisplaybreaks

\makeatletter
\def\subsection{\@startsection{subsection}{2}%
	\z@{.5\linespacing\@plus.7\linespacing}{.3\linespacing}%
	{\normalfont\bfseries}}
\makeatother

\theoremstyle{theorem}
\newtheorem{theorem}{Theorem}
\newtheorem{lemma}[theorem]{Lemma}

\theoremstyle{definition}

\theoremstyle{remark} \theoremstyle{question} \theoremstyle{example}

\newcommand{\N}{\mathbb{N}}   
\newcommand{\R}{\mathbb{R}}   
\newcommand{\U}{\mathbb{U}}
\newcommand{\V}{\mathbb{V}}
\newcommand{\W}{\mathbb{W}}

\newcommand{\cB}{\mathcal{B}}

\newcommand{\cK}{\mathcal{K}}

\newcommand{\cM}{\mathcal{M}}

\newcommand{\cU}{\mathcal{U}}

\newcommand{\eps}{\varepsilon}   
\newcommand{\ov} {\overline}     
\newcommand{\wt} {\widetilde}    

\newcommand{\diam}  {\operatorname{diam}}

\newcommand{\cantor}    {{\{0,1\}^\N}}

\begin{document}


\title[Uniformly Positive Entropy of Induced Transformations]
      {Uniformly Positive Entropy\\ of Induced Transformations}



\author{Nilson C. Bernardes Jr.}
\address{Departamento de Matem\'atica Aplicada, Instituto de Matem\'atica,
         Universidade Fede\-ral do Rio de Janeiro, Caixa Postal 68530,
         Rio de Janeiro, RJ, 21945-970, Brazil.}
\curraddr{}
\email{ncbernardesjr@gmail.com}
\thanks{}


\author{Udayan B. Darji}
\address{Department of Mathematics, University of Louisville, Louisville,
        KY 40208-2772, USA.}
\curraddr{}
\email{ubdarj01@louisville.edu}
\thanks{}


\author{R\^omulo M. Vermersch}
\address{Departamento de Matem\'atica, Centro de Ci\^encias F\'isicas e
         Matem\'aticas, Universidade Federal de Santa Catarina,
         Florian\'opolis, SC, 88040-900, Brazil.}
\curraddr{}
\email{romulo.vermersch@gmail.com}
\thanks{}


\subjclass[2010]{Primary 37B40; Secondary 28A33, 60B10, 54B20.}
\keywords{Continuous surjective maps, probability measures, Prohorov metric,
          weak$^*$ topology, hyperspaces, dynamics.}

\date{}

\dedicatory{}

\maketitle
	
\begin{abstract}
Let $(X,T)$ be a topological dynamical system consisting of a compact
metric space $X$ and a continuous surjective map $T : X \to X$.
By using local entropy theory, we prove that $(X,T)$ has uniformly positive
entropy if and only if so does the induced system $(\cM(X),\wt{T})$ on the
space of Borel probability measures endowed with the weak$^*$ topology. This
result can be seen as a version for the notion of uniformly positive entropy
of the corresponding result for topological entropy due to Glasner and Weiss.
\end{abstract}


\section{Introduction}\label{Intro}
	
Local entropy theory is a culmination of deep results in dynamics, ergodic
theory and combinatorics. Given a dynamical system with positive entropy,
it gives, in some sense, the location of where the entropy resides.
We refer the reader to the survey article by Glasner and Ye \cite{EGlaYe09}
for general information on the subject as well as important contributions
by many to the field. It is a powerful tool that can be applied in a variety
of settings. For example, the second author and Kato \cite{darjikato}, using
local entropy theory, settled some old problems concerning indecomposable
continua in dynamical systems. In this note we further show the power of
this theory by giving a simple proof of the theorem stated in the abstract.

By a {\em topological dynamical system} (TDS) we mean a pair $(X,T)$
consisting of a compact metric space $X$ with metric $d$ and a continuous
surjective map $T : X \to X$. Such a TDS induces, in a natural way, the TDSs
$$
(\cK(X),\ov{T}) \ \ \ \text{ and } \ \ \ (\cM(X),\wt{T}).
$$
In the first of these systems, $\cK(X)$ denotes the {\em hyperspace} of all
nonempty closed subsets of $X$ endowed with the {\em Hausdorff metric}
$$
d_H(K_1,K_2):= \inf\{\delta > 0 : K_1 \subset K_2^\delta \text{ and }
                                  K_2 \subset K_1^\delta\},
$$
where $A^\delta:= \{x \in X : d(x,A) < \delta\}$ is the
{\em $\delta$-neighborhood} of $A \subset X$, and
$\ov{T} : \cK(X) \to \cK(X)$ is the continuous surjective map given by
$$
\ov{T}(K):= T(K) \ \ \ (K \in \cK(X)).
$$
In the second one, $\cM(X)$ denotes the space of all Borel probability
measures on $X$ endowed with the {\em Prohorov metric}
$$
d_P(\mu,\nu):= \inf\{\delta > 0 : \mu(A) \leq \nu(A^\delta) + \delta
   \text{ and }  \nu(A) \leq \mu(A^\delta) + \delta \text{ for all } A \in \cB_X\},
$$
where $\cB_X$ is the $\sigma$-algebra of all Borel subsets of $X$, and
$\wt{T} : \cM(X) \to \cM(X)$ is the continuous surjective map given by
$$
(\wt{T}(\mu))(A):= \mu(T^{-1}(A)) \ \ \ (\mu \in \cM(X), A \in \cB_X).
$$
It is well known that $\cK(X)$ and $\cM(X)$ are compact metric spaces and
that $\ov{T}$ and $\wt{T}$ are homeomorphisms whenever $T$ is a 
homeomorphism. Moreover,
$$
d_P(\mu,\nu):= \inf\{\delta > 0 : \mu(A) \leq \nu(A^\delta) + \delta
   \text{ for all } A \in \cB_X\}
$$
(see p.\ 72 in \cite{PBil99}). We refer the reader to the books \cite{Nadler,AKec95} 
and \cite{PBil99,RDud02,AKec95} for a study of the spaces $\cK(X)$ and $\cM(X)$,
respectively.
	
An interesting line of investigation in the area of dynamical systems
is the study of the relations between the dynamics of the TDS $(X,T)$ and
the dynamics of the induced TDSs $(\cK(X),\ov{T})$ and $(\cM(X),\wt{T})$.
Such investigations were initiated by Bauer and Sigmund \cite{WBauKSig75},
and later were widely developed by several authors; see
\cite{AkiAusNag17,BerPerRod17,NBerRVer14,NBerRVer16,LFerCGoo16,EGlaBWei95,
GKLOP09,LiOprWu17,LiYanYe15,KSig78}, for instance.
It is worth to mention that there is a common line in these works: the search
for which dynamical properties the base system $(X,T)$ and the induced
systems $(\cK(X),\ov{T})$ and $(\cM(X),\wt{T})$ share with each other
and which dynamical properties they do not share at all.

It is well known that the extension of a TDS to the hyperspace can
dramatically increase the system's complexity. For instance, an example
was given in \cite{EGlaBWei93} (see also \cite{DKwiPOpr07}) of a
zero topologial entropy system $(X,T)$ whose hyperspace extension
$(\cK(X),\ov{T})$ has positive topological entropy. As another example,
it was proved in \cite{NBerUDar12,NBerRVer14} that for the generic
homeomorphism $h$ of the Cantor space $\cantor$, the TDS $(\cantor,h)$
has no Li-Yorke pair (in particular, it has zero topological entropy),
but its hyperspace extension $(\cK(\cantor),\ov{h})$ is uniformly
distributionally chaotic and has infinite topological entropy.

The situation for the extension $(\cM(X),\wt{T})$ is completely different,
at least from the point of view of topological entropy. Indeed, a deep and
surprising result due to Glasner and Weiss \cite{EGlaBWei95} asserts that
if $(X,T)$ has zero topological entropy, then so does $(\cM(X),\wt{T})$.
A corresponding result for the notion of null system was obtained by
Kerr and Li \cite{KerrLi05}: if $(X,T)$ is null, then so is $(\cM(X),\wt{T})$.
Recall that $(X,T)$ is said to be {\em null} if the topological sequence
entropy of $(X,T)$ is zero for any increasing sequence of natural numbers.
Recently, Qiao and Zhou \cite{QiaoZhou17} obtained such a result for the
notion of sequence entropy, which unified the above-mentioned results of
Glasner-Weiss and Kerr-Li. We observe that the converses of these results
are trivially true, since $(X,T)$ can be regarded as a subsystem of
$(\cM(X),\wt{T})$.

Our goal is to investigate the relationships between these
systems for the notion of {\it uniformly positive entropy} (UPE).
This notion was introduced by Blanchard \cite{Blan92} as a candidate for
an analogue in topological dynamics for the notion of a $K$-process in
ergodic theory. In fact, in that paper he proved that every non-trivial
factor of an UPE system has positive topological entropy and, shortly after,
he proved that an UPE system is disjoint from every minimal zero entropy
system \cite{Blan93}.
Although an UPE system is topologically weakly mixing but not always
strongly mixing \cite{Blan93}, Glasner and Weiss \cite{EGlaBWei94} proved
that UPE is a necessary condition for a TDS $(X,T)$ to have a $T$-invariant
probability measure $\mu$ of full support whose corresponding measurable
dynamical system $(X,T,\mu)$ is a $K$-process.

In this short note we prove that $(X,T)$ has UPE if and only if so does
$(\cM(X),\wt{T})$. This result can be seen as a version of the Glasner-Weiss
theorem \cite{EGlaBWei95} for the notion of uniformly positive entropy.
In the proof of this result we will use local entropy theory; more precisely, 
we will use the characterization of UPE by means of the notion of an
independence set given by Huang and Ye \cite{HuangYe} (see also Kerr and 
Li \cite{KerrLi07}).
Moreover, we will also use an important technique developed by Glasner and
Weiss \cite{EGlaBWei95} which connects linear operators to combinatorics.

We point out that Huang and Ye \cite{HuangYe} proved the corresponding
result for the hyperspace, i.e., a TDS $(X,T)$ has UPE if and only if
$(\cK(X),\ov{T})$ has UPE.


\section{Preliminaries}

Recall that the Prohorov metric $d_P$ on $\cM(X)$ induces the so-called
{\em weak$\,^*$ topology}, that is, the topology whose basic open
neighborhoods of $\mu \in \cM(X)$ are the sets of the form
$$
\V(\mu;f_1,\ldots,f_k;\eps):= \Big\{\nu \in \cM(X) :
  \Big| \int_X f_i \,d\nu - \int_X f_i \,d\mu \Big| < \eps
  \text{ for } i = 1,\ldots,k\Big\},
$$
where $k \geq 1$, $f_1,\ldots,f_k : X \to \R$ are continuous functions and
$\eps > 0$.
For each $n \in \N$, let
$$
\cM_n(X):= \Big\{\displaystyle\frac{1}{n}\sum_{i=1}^n \delta_{x_i}\in\cM(X) :
                 x_1,\ldots,x_n \in X \ \text{not necessarily distinct}\Big\},
$$
where $\delta_x$ denotes the unit mass concentrated at the point $x$ of $X$.
It is classical that $\bigcup_{n \in \N} \cM_n(X)$ is dense in $\cM(X)$.
Since $\cM_n(X)$ is $\wt{T}$-invariant, we can consider the TDS
$(\cM_n(X),\wt{T})$, where we are also denoting by $\wt{T}$ the
corresponding restricted map.
It will be convenient to denote $\cM(X)$ by $\cM_\infty(X)$.

Let us now recall some definitions and notations from entropy theory.
In what follows, all logarithms are in base $2$. Let $(X,T)$ be a TDS.
Given covers $\cU_1,\ldots,\cU_n$ of $X$, let
$$
\cU_1 \lor \cdots \lor \cU_n := \{U_1 \cap \cdots \cap U_n:
                                  U_1 \in \cU_1,\ldots,U_n \in \cU_n\}.
$$
The {\em entropy} of an open cover $\cU$ of $X$ is defined by
$$
H(\cU):= \log N(\cU),
$$
where $N(\cU)$ denotes the minimum cardinality of a subcover of $\cU$.
The {\em topological entropy of $T$ with respect to $\cU$} is defined by
$$
h_{top}(T,\cU):= \lim_{n \to \infty} \frac{1}{n} H(\cU^{n-1}),
$$
where $\cU^{n-1}:= \cU \lor T^{-1}\cU \lor\dots\lor T^{-(n-1)}\cU$,
and the {\em topological entropy} of $T$ is given by
$$
h_{top}(T):= \sup_{\cU} h_{top}(T,\cU),
$$
where the supremum is taken over all open covers of $X$.
The notion of topological entropy was introduced by Adler, Konheim and
McAndrew \cite{AdlKonMcA65}. It plays a fundamental role in topological
dynamics and its applications. Finally, an open cover $\cU = \{U,V\}$ of $X$
consisting of two sets is called a {\it standard cover} if both $U$ and $V$
are non-dense in $X$. The TDS $(X,T)$ is said to have {\em uniformly positive
entropy} (UPE) if $h_{top}(T,\cU) > 0$ for every standard cover $\cU$ of $X$.
The notion of UPE is due to Blanchard \cite{Blan92}.

The following useful characterization of UPE was given by Huang and Ye
\cite[Theorem~7.4(i)]{HuangYe} (see also Kerr and Li \cite{KerrLi07}).

\medskip
\noindent {\bf Theorem 0.} {\it A TDS $(X,T)$ has UPE if and only if
every pair $(U,V)$ of disjoint nonempty open sets in $X$ has an
independence set of positive density.}

\medskip
Recall that a set $I \subset \N$ is said to be an {\em independence set}
for a tuple $(A_1,A_2,\dots,A_k )$ of subsets of $X$ if
$$
\bigcap_{j \in J} T^{-j}(A_{\sigma(j)}) \neq \emptyset
$$
for every nonempty finite subset $J$ of $I$ and every function
$\sigma : J \to \{1,2,\dots, k\}$. Recall also that a subset $I$ of $\N$
has {\em positive density} if the limit
$$
\lim_{n \to \infty} \frac{|I \cap \{1,\dots,n\}|}{n}
$$
exists and is nonzero, where $|K|$ denotes the cardinality of the subset
$K$ of $\N$.


\section{Proof of the Main Result}

\smallskip
\begin{lemma}\label{newbasis}
The sets of the form
\begin{equation}\label{Equa1}
\W(U_1,\ldots,U_k;\eta_1,\ldots,\eta_k):=
  \{\nu \in \cM(X) : \nu(U_i) > \eta_i \text{ for } i = 1,\ldots,k\},
\end{equation}
where $k \geq 1$, $U_1,\ldots,U_k$ are nonempty disjoint open sets in $X$
and $\eta_1,\ldots,\eta_k$ are positive real numbers with
$\eta_1+\cdots+\eta_k < 1$, form a basis for the weak$\,^*$ topology on
$\cM(X)$.
\end{lemma}

\begin{proof}
Suppose that $\mu \in \W:= \W(U_1,\ldots,U_k;\eta_1,\ldots,\eta_k)$.
By inner regularity, there are compact sets $C_1,\ldots,C_k$ in $X$
such that $C_i \subset U_i$ and $\mu(C_i) > \eta_i$ for all $i$.
Let $\delta > 0$ be so small that $C_i^\delta \subset U_i$ and
$\mu(C_i) - \delta > \eta_i$ for all $i$. Then $d_P(\nu,\mu) < \delta$
implies
$$
\nu(U_i) \geq \nu(C_i^\delta) \geq \mu(C_i) - \delta > \eta_i
  \ \ \text{ for all } i \in \{1,\ldots,k\},
$$
that is, $\nu \in \W$. This proves that $\W$ is open in $\cM(X)$.

Now, take an open ball $B_{d_P}(\mu;\eps)$ in $\cM(X)$ and choose
$\delta \in (0,\eps)$. We claim that there exist disjoint open sets
$U_1,\ldots,U_k$ in $X$ such that
\begin{equation}\label{Equa2}
\diam U_i < \delta \ \ \text{ for all } i \in \{1,\ldots,k\}
\end{equation}
and
\begin{equation}\label{Equa3}
\mu(U_1) + \cdots + \mu(U_k) > 1 - \delta/2.
\end{equation}
Indeed, let $\{V_1,\ldots,V_k\}$ be an open cover of $X$ by sets with
diameters $< \delta$. Let $C_0:= \varnothing$ and let $C_1,\ldots,C_k$ be
compact sets in $X$ such that
$$
C_i \subset V_i \backslash (C_0 \cup \ldots \cup C_{i-1}) \ \text{ and } \ 
\mu\big(V_i \backslash (C_1 \cup \ldots \cup C_i)\big) < \frac{\delta}{2k}
\ \ \ \ (i \in \{1,\ldots,k\}).
$$
Since $C_1,\ldots,C_k$ are disjoint and since (\ref{Equa2}) and (\ref{Equa3})
hold with $C_1,\ldots,C_k$ instead of $U_1,\ldots,U_k$, it is clear that
there exist $U_1,\ldots,U_k$ with the desired properties.
Moreover, we may assume $\mu(U_i) > 0$ for all $i$.
Let $Y:= X \backslash (U_1 \cup \ldots \cup U_k)$.
For each $i$, take $\eta_i > 0$ with
$\mu(U_i) - \frac{\delta}{2k} < \eta_i < \mu(U_i)$.
We shall prove that
$$
\mu \in \W(U_1,\ldots,U_k;\eta_1,\ldots,\eta_k) \subset B_{d_P}(\mu;\eps).
$$
It is clear that $\mu \in \W(U_1,\ldots,U_k;\eta_1,\ldots,\eta_k)$.
Pick $\nu \in \W(U_1,\ldots,U_k;\eta_1,\ldots,\eta_k)$ and $A \in \cB_X$.
Let $I:= \{1 \leq i \leq k : A \cap U_i \neq \varnothing\}$. Since
$\mu(Y) < \delta/2$ and $A^\delta \cap U_i = U_i$ whenever $i \in I$, we get
\begin{align*}
\nu(A^\delta) &\geq \sum_{i \in I} \nu(A^\delta \cap U_i)
               = \sum_{i \in I} \nu(U_i)
               > \sum_{i \in I} \eta_i
               > \sum_{i \in I} \mu(A \cap U_i) - \frac{\delta}{2}\\
              &= \mu(A) - \mu(A \cap Y) - \frac{\delta}{2}
               > \mu(A) - \delta.
\end{align*}
This proves that $d_P(\nu,\mu) \leq \delta < \eps$.
\end{proof}

\begin{lemma}\label{lastbaselemma}
Given nonempty open sets $\U_0,\U_1$ in $\cM_n(X)$, where
$1 \leq n \leq \infty$, there exist $m \in \N$ and nonempty open sets
$U_{0,1},\ldots,U_{0,m},U_{1,1},\ldots,U_{1,m}$ in $X$ such that
\begin{equation}\label{EquaA}
\frac{1}{m} \sum_{i=1}^m \delta_{x_{k,i}} \in \U_k \ \text{ whenever }
  x_{k,1} \in U_{k,1},\ldots,x_{k,m} \in U_{k,m} \ \ \ \ (k \in \{0,1\}).
\end{equation}
Moreover, we can take $m = n$ if $n < \infty$.
\end{lemma}

\begin{proof}
By Lemma~\ref{newbasis}, there is an open set $\W_k$ of the form
$\W(V_{k,1},\ldots,V_{k,j_k};\eta_{k,1},\ldots,\eta_{k,j_k})$ in $\cM(X)$ such that
\begin{equation}\label{EquaB}
\varnothing \neq \W_k \cap \cM_n(X) \subset \U_k \ \ \ (k \in \{0,1\}).
\end{equation}

Assume $n < \infty$ and take
\begin{equation}\label{EquaC}
\mu_k:= \frac{1}{n} \sum_{i=1}^n \delta_{y_{k,i}} \in \W_k
  \ \ \text{ for } k \in \{0,1\}.
\end{equation}
Put $m:= n$. For each $k \in \{0,1\}$ and each $i \in \{1,\ldots,n\}$, 
if $y_{k,i} \in V_{k,j}$ for some (necessarily unique) $j \in \{1,\ldots,j_k\}$, 
let $U_{k,i}$ be an open neighborhood of $y_{k,i}$ with $U_{k,i} \subset V_{k,j}$,
and if $y_{k,i} \not\in V_{k,1} \cup \ldots \cup V_{k,j_k}$, put $U_{k,i}:= X$.
We claim that (\ref{EquaA}) holds. Indeed, take 
$x_{k,1} \in U_{k,1},\ldots,x_{k,n} \in U_{k,n}$. By (\ref{EquaB}), it is enough to
show that
\begin{equation}\label{EquaC2}
\frac{1}{n} \sum_{i=1}^n \delta_{x_{k,i}} \in \W_k.
\end{equation}
Let $r_{k,j}:= \big|\{1 \leq i \leq n : y_{k,i} \in V_{k,j}\}\big|$ and 
$s_{k,j}:= \big|\{1 \leq i \leq n : x_{k,i} \in V_{k,j}\}\big|$.
Since $y_{k,i} \in V_{k,j}$ implies $x_{k,i} \in U_{k,i} \subset V_{k,j}$, we have that
$r_{k,j} \leq s_{k,j}$. Thus, by (\ref{EquaC}), 
$$
\left(\frac{1}{n} \sum_{i=1}^n \delta_{x_{k,i}}\right)\big(V_{k,j}\big) 
= \frac{s_{k,j}}{n} \geq \frac{r_{k,j}}{n} = \mu_k(V_{k,j}) > \eta_{k,j},
$$
for every $1 \leq j \leq j_k$, which proves (\ref{EquaC2}).

Now, assume $n = \infty$. Since $\bigcup_{t \in \N} \cM_t(X)$
is dense in $\cM(X)$, we can take
\begin{equation}\label{EquaD}
\nu_k:= \frac{1}{t_k} \sum_{i=1}^{t_k} \delta_{z_{k,i}} \in \W_k
  \ \ \text{ for } k \in \{0,1\}.
\end{equation}
For each $i \in \{1,\ldots,t_k\}$, let $A_{k,i}$ be an open neighborhood
of $z_{k,i}$ with $A_{k,i} \subset V_{k,j}$ if $z_{k,i} \in V_{k,j}$
(put $A_{k,i}:= X$ if $z_{k,i} \not\in V_{k,1} \cup \ldots \cup V_{k,j_k}$).
Put $m:= t_0 t_1$ and produce a sequence $U_{0,1},\ldots,U_{0,m}$
(resp.\ $U_{1,1},\ldots,U_{1,m}$) by putting each $A_{0,i}$ (resp.\ each
$A_{1,i}$) exactly $t_1$ (resp.\ $t_0$) times. By reasoning as in the previous
paragraph, we obtain (\ref{EquaA}) from (\ref{EquaB}) and (\ref{EquaD}).
\end{proof}

We will also need the following result \cite[Proposition~2.1]{EGlaBWei95}:

\begin{lemma}\label{BS}
Given constants $\eps > 0$ and $b > 0$, there exist constants $m_0 \in \N$
and $c > 0$ such that the following property holds for every $m \geq m_0$:
if $\varphi : \ell_1^k \to \ell_\infty^m$ is a linear map with
$\|\varphi\| \leq 1$, and if $\varphi(B_{\ell_1^k})$ contains more than
$2^{bm}$ points that are $\eps$-separated, then $k \geq 2^{cm}$.
\end{lemma}

We observe that $\ell_1^k$ denotes the vector space $\R^k$ endowed with
the $\ell_1$-norm, that is,
$\|(r_1,\ldots,r_k)\|:= |r_1| + \cdots + |r_k|$,
and that $\ell_\infty^m$ denotes the vector space $\R^m$ endowed with the
$\ell_\infty$-norm, that is,
$\|(s_1,\ldots,s_m)\|:= \max\{|s_1|,\ldots,|s_m|\}$.
Moreover, $B_{\ell_1^k}$ denotes the closed unit ball of the Banach space
$\ell_1^k$.

\begin{theorem}\label{M(X)Mainresult}
For every TDS $(X,T)$, the following assertions are equivalent:
\begin{itemize}
\item [\rm (i)]   $(X,T)$ has UPE;.
\item [\rm (ii)]  $(\cM_n(X),\wt{T})$ has UPE for some $1 \leq n \leq \infty$;
\item [\rm (iii)] $(\cM_n(X),\wt{T})$ has UPE for every $1 \leq n \leq \infty$.
\end{itemize}                 
\end{theorem}
	
\begin{proof}
(ii) $\Rightarrow$ (i): Let us consider first the case $n = \infty$.
Suppose that $(\cM(X),\wt{T})$ has UPE. Let $\cU = \{U_0,U_1\}$ be
a standard cover of $X$. We have to prove that $h_{top}(T,\cU) > 0$.
For this purpose, let $V_0$ and $V_1$ be nonempty open sets in $X$ with
$$
V_0 \subset U_0 \backslash \ov{U_1}, \ \ V_1 \subset U_1 \backslash \ov{U_0}
\ \ \text{ and } \ \ \ov{V_0} \cap \ov{V_1} = \varnothing.
$$
Define
$$
\V_0:= \{\mu \in \cM(X) : \mu(V_0) > 0.9\} \ \text{ and } \
\V_1:= \{\mu \in \cM(X) : \mu(V_1) > 0.9\},
$$
which are disjoint nonempty open sets in $\cM(X)$. By Theorem~0, the pair
$(\V_0,\V_1)$ has an independence set $I \subset \N$ of positive density.
Let $b > 0$ and $m_1 \in \N$ be such that
$$
\frac{|I \cap \{1,\ldots,m\}|}{m} > b \ \ \text{ whenever } m \geq m_1.
$$
Put $\eps = 0.5$ and let $m_0 \in \N$ and $c > 0$ be the constants
associated to $\eps$ and $b$ according to Lemma~\ref{BS}.
Fix $m \geq \max\{m_0,m_1\}$ and choose a subcover
$\{A_1,A_2,\ldots,A_{k_m}\}$ of $\cU^{m-1}$ with minimum cardinality.
Let
$$
B_1:= A_1 \ \text{ and } \
B_i:= A_i \backslash (A_1 \cup \ldots \cup A_{i-1}) \
\text{ for } 2 \leq i \leq k_m.
$$
As $\{A_1,A_2,\ldots,A_{k_m}\}$ is minimal, we have that
$B_i \neq \varnothing$ for every $i$. Let
$$
M:= \big[t_{i,j}\big]_{1 \leq i \leq k_m, 0 \leq j \leq m-1}
$$
be a $k_m \times m$ matrix of 0's and 1's such that
$$
B_i \subset U_{t_{i,0}} \cap T^{-1}(U_{t_{i,1}}) \cap T^{-2}(U_{t_{i,2}})
  \cap \ldots \cap T^{-(m-1)}(U_{t_{i,m-1}}),
$$
for all $1 \leq i \leq k_m$. Consider the linear map
$\varphi : \ell_1^{k_m} \to \ell_\infty^m$ given by
$$
\varphi(r_1,\ldots,r_{k_m}):= [r_1 \ \cdots \ r_{k_m}]\, M.
$$
Clearly, $\|\varphi\| \leq 1$. Let $J:= I \cap \{1,\ldots,m\}$.
Since $I$ is an independence set for $(\V_0,\V_1)$, for each 
$\sigma: J \to \{0,1\}$, there exists $\mu_\sigma \in \cM(X)$ such that
$$
\wt{T}^j(\mu_\sigma) \in \V_{\sigma(j)} \ \ \text{ for all } j \in J.
$$
Let $\sigma, \sigma' : J \to \{0,1\}$ be distinct functions and let
$s \in J$ be such that $\sigma(s) \neq \sigma'(s)$.
Let us assume $\sigma(s) = 1$ and $\sigma'(s) = 0$. Then,
$$
\mu_\sigma(T^{-s}(V_1)) > 0.9 \ \ \text{ and } \ \
\mu_{\sigma'}(T^{-s}(V_0)) > 0.9.
$$
Since
$$
T^{-s}(V_1) \subset \bigcup \{B_i : t_{i,s} = 1\} \subset T^{-s}(U_1)
$$
(for the first inclusion use the fact that $T^{-s}(V_1)$ and $T^{-s}(U_0)$
are disjoint), we obtain
$$
\mu\big(T^{-s}(V_1)\big) \leq \sum_{i=1}^{k_m} t_{i,s}\, \mu(B_i)
                         \leq \mu\big(T^{-s}(U_1)\big)
                         \ \ \ \ (\mu \in \cM(X)).
$$                         
Hence, the $s^\text{th}$ coordinates of the points
$$
\varphi\big(\mu_\sigma(B_1),\ldots,\mu_\sigma(B_{k_m})\big)
\ \ \text{ and } \ \
\varphi\big(\mu_{\sigma'}(B_1),\ldots,\mu_{\sigma'}(B_{k_m})\big)
$$
are greater than $0.9$ and smaller than $0.1$, respectively, showing that
these points are $\eps$-separated. Since $|J| > bm$, there
are more than $2^{bm}$ functions $\sigma$. This shows that
$\varphi(B_{\ell_1^{k_m}})$ contains more than $2^{bm}$ points that
are $\eps$-separated. Thus, by Lemma~\ref{BS}, $k_m \geq 2^{cm}$.
This implies that $h_{top}(T,\cU) \geq c > 0$, as desired. 

Now, suppose that $(\cM_n(X),\wt{T})$ has UPE for some $2 \leq n < \infty$.
Given disjoint nonempty open sets $U_0,U_1$ in $X$, consider
$$
\U_0:= \Big\{\mu \in \cM_n(X) : \mu(U_0) > \frac{n-1}{n}\Big\}
\ \text{ and } \
\U_1:= \Big\{\mu \in \cM_n(X) : \mu(U_1) > \frac{n-1}{n}\Big\},
$$
which are disjoint nonempty open sets in $\cM_n(X)$.
Since $(\cM_n(X),\wt{T})$ has UPE, the pair $(\U_0,\U_1)$ has an
independence set $I \subset \N$ of positive density. Hence, given
$J \subset I$ nonempty and finite, and given $\sigma : J \to \{0,1\}$,
we have that
$$
\bigcap_{j \in J} \wt{T}^{-j}(\U_{\sigma(j)}) \neq \varnothing.
$$
If $\mu:= \frac{1}{n} \sum_{i=1}^n \delta_{x_i}$ belongs to the above
intersection, then
$$
\frac{1}{n} \sum_{i=1}^n \delta_{T^j x_i} \in \U_{\sigma(j)}
\ \ \text{ for all } j \in J.
$$
By the definitions of $\U_0$ and $\U_1$, this implies that
$$
\{T^jx_1,\ldots,T^jx_n\} \subset U_{\sigma(j)} \ \ \text{ for all } j \in J.
$$
Thus,
$$
\{x_1,\ldots,x_n\} \subset \bigcap_{j \in J} T^{-j}(U_{\sigma(j)}),
$$
proving that the above intersection is nonempty. Thus, $I$ is an
independence set of positive density for the pair $(U_0,U_1)$.
By Theorem~0, $(X,T)$ has UPE.

\smallskip
\noindent
(i) $\Rightarrow$ (iii): Suppose that $(X,T)$ has UPE.
Fix $1 \leq n \leq \infty$ and let $\U_0,\U_1$ be disjoint nonempty open
sets in $\cM_{n}(X)$. By Lemma~\ref{lastbaselemma}, there exist an integer
$m \geq 1$ and nonempty open sets
$U_{0,1},\ldots,U_{0,m},U_{1,1},\ldots,U_{1,m}$ in $X$ such that
\begin{equation}\label{E1}
R_m(U_{0,1} \times\cdots\times U_{0,m}) \subset \U_0
\ \text{ and } \
R_m(U_{1,1} \times\cdots\times U_{1,m}) \subset \U_1,
\end{equation}
where $R_m(x_1,\ldots,x_m):= \frac{1}{m} \sum_{i=1}^m \delta_{x_i}$.
Moreover, we may assume $m = n$ if $n < \infty$.
Let $X_m:= X \times \cdots \times X$ ($m$ times) and
$T_m:= T \times \cdots \times T$ ($m$ times).
Note that $\wt{T} \circ R_m = R_m \circ T_m$.
Since $\U_0 \cap \U_1 = \varnothing$, (\ref{E1}) implies that
\begin{equation}\label{E2}
(U_{0,1} \times\cdots\times U_{0,m}) \cap (U_{1,1} \times\dots\times U_{1,m})
 = \varnothing.
\end{equation}
Since any finite product of UPE systems is a UPE system \cite{EGla97},
the TDS $(X_m,T_m)$ has UPE. Hence, by (\ref{E2}) and Theorem~0,
there exists $I \subset \N$ of positive density such that
$$
\bigcap_{j \in J} T_m^{-j}(U_{\sigma(j),1} \times\cdots\times U_{\sigma(j),m})
 \neq \varnothing
$$
for every nonempty finite subset $J$ of $I$ and every function
$\sigma : J \to \{0,1\}$. For such a $J$ and such a $\sigma$, we have that
\begin{align*}
\bigcap_{j \in J} \wt{T}^{-j}(\U_{\sigma(j)})
  &\supset \bigcap_{j \in J} \wt{T}^{-j}\big(R_m(U_{\sigma(j),1}
                                   \times\cdots\times U_{\sigma(j),m})\big)\\
  &\supset \bigcap_{j \in J} R_m\big(T_m^{-j}(U_{\sigma(j),1}
                                   \times\cdots\times U_{\sigma(j),m})\big)\\
  &\supset R_m\Big(\bigcap_{j \in J} T_m^{-j}(U_{\sigma(j),1}
                                   \times\dots\times U_{\sigma(j),m})\Big)
   \neq\varnothing.
\end{align*}
This shows that $I$ is an independence set of positive density for the pair
$(\U_0,\U_1)$. Thus, by Theorem~0, $(\cM_n(X),\wt{T})$ has UPE.
\end{proof}

\section{Acknowledgment}
We would like to thank the referee for valuable suggestions which improved the exposition of the article. 

\end{document}